\documentclass[12pt,english]{article}
\usepackage[T1]{fontenc}
\usepackage[latin9]{inputenc}
\usepackage{amsmath}
\usepackage{amssymb}

\makeatletter
\usepackage[T1]{fontenc}
\usepackage[latin9]{inputenc}
\usepackage{color}
\usepackage{amsmath}
\usepackage{amssymb}
\usepackage{mathrsfs}
\usepackage{babel}
\usepackage[active]{srcltx}
\usepackage[usenames,dvipsnames]{pstricks}
\usepackage{epsfig}
\usepackage{pst-grad} 
\usepackage{pst-plot} 

\makeatletter

\providecommand{\U}[1]{\protect\rule{.1in}{.1in}}
\providecommand{\U}[1]{\protect\rule{.1in}{.1in}}
\newtheorem{theorem}{Theorem}

\newtheorem{definition}[theorem]{Definition}

\newtheorem{proposition}[theorem]{Proposition}
\newtheorem{remark}[theorem]{Remark}

\newenvironment{proof}[1][Proof]{\noindent\textbf{#1.} }{\ \rule{0.5em}{0.5em}}

\makeatother

\makeatother

\usepackage{babel}

\begin{document}

\title{The minimal context for local boundedness in topological vector spaces}

\author{M.D. Voisei}

\date{{}}
\maketitle
\begin{abstract}
The local boundedness of classes of operators is analyzed on different
subsets directly related to their Fitzpatrick functions and characterizations
of the topological vector spaces for which that local boundedness
holds is given in terms of the uniform boundedness principle. For
example the local boundedness of a maximal monotone operator on the
algebraic interior of its domain convex hull is a characteristic of
barreled locally convex spaces. 
\end{abstract}

\section{Introduction}

The local boundedness of a monotone operator defined on an open set
of a Banach space was first intuited by Kato in \cite{MR0238135}
while performing a comparison of (sequential) demicontinuity and hemicontinuity.
Under a Banach space settings, the first result concerning the local
boundedness of monotone operators appears in 1969 and is due to Rockafellar
\cite[Theorem\ 1, p. 398]{MR0253014}. In 1972 in \cite{MR0312336},
the local boundedness of monotone-type operators is proved under a
Fr\'echet space context. In 1988 the local boundedness of a monotone
operator defined in a barreled normed space is proved in \cite{MR995141}
on the algebraic interior of the domain. The authors of \cite{MR995141}
call their assumptions {}``minimal'' but they present no argument
about the minimality of their hypotheses or in what sense that minimality
is to be understood.

Our principal aim, in Theorems \ref{LBG}, \ref{LBM}, \ref{LBMax},
\ref{LBM-conv}, \ref{LBMax-conv},  \ref{LBM-conv-plus} below, is
to show that the context assumptions in \cite[Theorem\ 2]{MR995141},
\cite{MR0312336}, \cite[Theorem 1]{MR0253014} are not minimal and
to characterize topological vector spaces that offer the proper context
for an operator to be locally bounded, for example, on the algebraic
interior of its domain convex hull. 

The plan of the paper is as follows. In the next section we introduce
the main notions and notations followed by a study of the so called
Banach-Steinhaus property. The main object of Sections 3 and 4 is
to provide characterizations of the topological vector spaces on which
the local boundedness of an operator holds on different subsets directly
related to the operator via its Fitzpatrick function.

\section{Preliminaries}

In this paper the conventions $\sup\emptyset=-\infty$ and $\inf\emptyset=\infty$
are enforced.

Given $(E,\mu)$ a real topological vector space (TVS for short) and
$A\subset E$ we denote by {}``$\operatorname*{conv}A$'' the \emph{convex
hull} of $A$, {}``$\operatorname*{cl}_{\mu}(A)=\overline{A}^{\mu}$''
the $\mu-$\emph{closure} of $A$, {}``$\operatorname*{int}_{\mu}A$''
the $\mu-$\emph{topological interior }of $A$, {}``$\operatorname*{core}A$''
the \emph{algebraic interior} of $A$. The use of the  $\mu-$notation
is avoided whenever the topology $\mu$ is implicitly understood.

We denote by $\iota_{A}$ the \emph{indicator} \emph{function} of
$A\subset E$ defined by $\iota_{A}(x):=0$ for $x\in A$ and $\iota_{A}(x):=\infty$
for $x\in E\setminus A$. 

For $f,g:E\rightarrow\overline{\mathbb{R}}$ we set $[f\leq g]:=\{x\in E\mid f(x)\leq g(x)\}$;
the sets $[f=g]$, $[f<g]$, and $[f>g]$ being defined in a similar
manner.

Throughout this paper, if not otherwise explicitly mentioned, $(X,\tau)$
is a non-trivial (that is, $X\neq\{0\}$) TVS, $X^{\ast}$ is its
topological dual endowed with the weak-star topology $w^{\ast}$,
the topological dual of $(X^{\ast},w^{\ast})$ is identified with
$X$ and the weak topology on $X$ is denoted by $w$. The \emph{duality
product} of $X\times X^{\ast}$ is denoted by $\left\langle x,x^{\ast}\right\rangle :=x^{\ast}(x)=:c(x,x^{\ast})$,
for $x\in X$, $x^{\ast}\in X^{\ast}$.

The class of neighborhoods of $x\in X$ in $(X,\tau)$ is denoted
by $\mathscr{V}_{\tau}(x)$. 

As usual, with respect to the dual system $(X,X^{*})$, for $A\subset X$,
the \emph{orthogonal of} $A$ is $A^{\perp}:=\{x^{*}\in X^{*}\mid\langle x,x^{*}\rangle=0,\ \forall x\in A\}$,
the \emph{polar of} $A$ is $A^{\circ}:=\{x^{*}\in X^{*}\mid|\langle x,x^{*}\rangle|\le1,\ \forall x\in A\}$,
the support function of $A$ is $\sigma_{A}(x^{*}):=\sup_{x\in A}\langle x,x^{*}\rangle$,
$x^{*}\in X^{*}$ while for $B\subset X^{*}$, the orthogonal of $B$
is $B^{\perp}:=\{x\in X\mid\langle x,x^{*}\rangle=0,\ \forall x^{*}\in B\}$,
the polar of $B$ is $B^{\circ}:=\{x\in X\mid|\langle x,x^{*}\rangle|\le1,\ \forall x^{*}\in B\}$,
and the support function of $B$ is $\sigma_{B}(x):=\sup_{x^{*}\in B}\langle x,x^{*}\rangle$,
$x\in X$.

To a multi\-function $T:X\rightrightarrows X^{\ast}$ we associate
its \emph{graph:} $\operatorname*{Graph}T=\{(x,x^{\ast})\in X\times X^{\ast}\mid x^{\ast}\in Tx\}$,
\emph{domain}: $D(T):=\{x\in X\mid Tx\neq\emptyset\}=\operatorname*{Pr}_{X}(\operatorname*{Graph}T)$,
and \emph{range}: $R(T):=\{x^{*}\in X^{*}\mid x^{*}\in T(x),\ \mathrm{for\ some}\ x\in X\}=\operatorname*{Pr}_{X^{*}}(\operatorname*{Graph}T)$.
Here $\operatorname*{Pr}_{X}$ and $\operatorname*{Pr}_{X^{*}}$ are
the projections of $X\times X^{*}$ onto $X$ and $X^{\ast}$, respectively.
When no confusion can occur, $T$ will be identified with $\operatorname*{Graph}T$.

The \emph{Fitzpatrick function }associated to $T:X\rightrightarrows X^{*}$,
$\varphi_{T}:X\times X^{*}\rightarrow\overline{\mathbb{R}}$ is given
by (see \cite{MR1009594})\[
\varphi_{T}(x,x^{*}):=\sup\{\langle a,x^{*}\rangle+\langle x-a,a^{*}\rangle\mid a^{*}\in Ta\},\ (x,x^{*})\in X\times X^{*}.\]
 Accordingly, for every $\epsilon\in\mathbb{R}$, the set $T_{\epsilon}^{+}:=[\varphi_{T}\le c+\epsilon]$
describes all $(x,x^{*})\in X\times X^{*}$ that are $\epsilon-$monotonically
related (m.r. for short) to $T$, that is $(x,x^{*})\in[\varphi_{T}\le c+\epsilon]$
iff $\langle x-a,x^{*}-a^{*}\rangle\ge-\epsilon$, for every $(a,a^{*})\in T$.

For every $\epsilon\ge0$, we consider on a TVS $(X,\tau)$ the following
classes of functions and operators 
\begin{description}
\item [{$\Lambda(X)$}] the class formed by proper convex functions $f:X\rightarrow\overline{\mathbb{R}}$.
Recall that $f$ is \emph{proper} if $\mathrm{dom}\: f:=\{x\in X\mid f(x)<\infty\}$
is nonempty and $f$ does not take the value $-\infty$, 
\item [{$\Gamma_{\tau}(X)$}] the class of functions $f\in\Lambda(X)$
that are $\tau$--lower semi-continuous (\emph{$\tau$--}lsc for short), 
\item [{$\mathcal{M}_{\epsilon}(X)$}] the class of non-empty $\epsilon-$monotone
operators $T:X\rightrightarrows X^{*}$. Recall that $T:X\rightrightarrows X^{*}$
is $\epsilon-$\emph{monotone} if $\left\langle x_{1}-x_{2},x_{1}^{\ast}-x_{2}^{\ast}\right\rangle \ge-\epsilon$,
for all $(x_{1},x_{1}^{*}),(x_{2},x_{2}^{*})\in T$, 
\item [{\textmd{$\mathcal{M}_{\epsilon}^{+}(X):=\{T_{\epsilon}^{+}\mid T\in\mathcal{M}_{\epsilon}(X)\}$,}}]~
\item [{$\mathfrak{M}_{\epsilon}(X)$}] the class of $\epsilon-$maximal
monotone operators $T:X\rightrightarrows X^{*}$. The maximality is
understood in the sense of graph inclusion as subsets of $X\times X^{*}$,
\item [{\textmd{$\mathscr{M}_{\infty}(X):=\bigcup_{\varepsilon\ge0}\mathcal{M}_{\varepsilon}(X)=\{T:X\rightrightarrows X^{*}\mid\inf_{z,w\in T}c(z-w)\neq\pm\infty\}$,}}]~
\item [{\textmd{the}}] \emph{$\epsilon-$subdifferential} of $f$ at $x\in X$:
$\partial_{\epsilon}f(x):=\{x^{\ast}\in X^{\ast}\mid\left\langle x^{\prime}-x,x^{\ast}\right\rangle +f(x)\leq f(x^{\prime})+\epsilon,\ \forall x^{\prime}\in X\}$
for $x\in\operatorname*{dom}f$; $\partial_{\epsilon}f(x):=\emptyset$
for $x\not\in\operatorname*{dom}f$,
\item [{\textmd{$\mathscr{G}_{\epsilon}(X):=\{\partial_{\epsilon}f\mid f\in\Gamma_{\tau}(X)\}$,}}] $\mathfrak{B}(X):=\{\partial\sigma_{B}\mid B\subset X^{*}$
is $w^{*}-$bounded$\}$. 
\end{description}
For $\epsilon=0$, the use of the $\epsilon-$notation is avoided.

\begin{definition}\emph{ \label{def-lb} Let $(X,\tau)$ be a TVS.
A multi-function $T:X\rightrightarrows X^{*}$ is }locally bounded
at\emph{ $x_{0}\in X$ if there exists $U\in\mathscr{V}_{\tau}(x_{0})$
such that $T(U):=\cup_{x\in U}Tx$ is an equicontinuous subset of
$X^{*}$; }locally bounded on\emph{ $S\subset X$ if $T$ is locally
bounded at every $x\in S$. The local boundedness of $T:X\rightrightarrows X^{*}$
is interesting only at $x_{0}\in\overline{D(T)}$ ($\tau-$closure)
since for every $x_{0}\not\in\overline{D(T)}$ there is $U\in\mathscr{V}_{\tau}(x_{0})$
such that $T(U)$ is void. Consequently, $T$ is locally bounded outside
$\overline{D(T)}$. }\end{definition}

Given a TVS $(X,\tau)$ with topological dual $X^{*}$, a set $B\subset X^{*}$
is:

\medskip

$\bullet$ \emph{pointwise-bounded} if $B_{x}:=\{x^{*}(x)\mid x^{*}\in B\}$
is bounded in $\mathbb{R}$, for every $x\in X$ or, equivalently,
$B$ is $w^{*}-$bounded in $X^{*}$;

$\bullet$ ($\tau-$)\emph{equicontinuous }if for every $\epsilon>0$
there is $V_{\epsilon}\in\mathscr{V}_{\tau}(0)$ such that $x^{*}(V_{\epsilon})\subset(-\epsilon,\epsilon)$,
for every $x^{*}\in B$, or, equivalently, $B$ is contained in the
polar $V^{\circ}$ of some (symmetric) $V\in\mathscr{V}_{\tau}(0)$. 

\medskip

We say that a topological vector space $(X,\tau)$ has the \emph{Banach-Steinhaus
property} if every pointwise-bounded subset of $X^{*}$ is equicontinuous.

\begin{theorem} \label{BS char} Let $(X,\tau)$ be a TVS. The following
are equivalent

\medskip

\emph{(i) }$(X,\tau)$ has the Banach-Steinhaus property.

\medskip

\emph{(ii) }Every absorbing, convex, and weakly-closed subset of $X$
is a $\tau-$neigh\-bor\-hood of $0\in X$. 

\medskip

\emph{(iii) }Every $f\in\Gamma_{w}(X)$ is $\tau-$continuous on $\operatorname*{int}_{\tau}(\operatorname*{dom}f)$
(or, equivalently, $f$ is bounded above on a $\tau-$neighborhood
of some $x\in\operatorname*{dom}f$). In this case, for every $f\in\Gamma_{\tau}(X)$,
$\operatorname*{int}_{\tau}(\operatorname*{dom}f)=\operatorname*{core}(\operatorname*{dom}f)$. 

\end{theorem} 

\begin{proof} (i) $\Rightarrow$ (ii) Let $C\subset X$ be absorbing,
convex, and weakly-closed. Then $C^{\circ}$ is pointwise-bounded
and equicontinuous in $X^{*}$ due to the Banach-Steinhaus property.
Therefore $C^{\circ}\subset V^{\circ}$, for some $V\in\mathscr{V}_{\tau}(0)$
followed by $V\subset V^{\circ\circ}\subset C^{\circ\circ}=C$ , due
to the Bipolar Theorem, and so $C\in\mathscr{V}_{\tau}(0)$. 

(ii) $\Rightarrow$ (iii) Let $f\in\Gamma_{w}(X)$, $x_{0}\in\operatorname*{core}(\operatorname*{dom}f)$,
and $a>f(x_{0})$. The level set $[f\le a]$ is weakly-closed and
convex.

For every $x\in X$, let $\mu>0$ be such that $\mu x\in\operatorname*{dom}f-x_{0}$.
Therefore, for every $0\le\lambda\le1$, $f(x_{0}+\lambda\mu x)=f(\lambda(x_{0}+\mu x)+(1-\lambda)x_{0})$$\le\lambda f(x_{0}+\mu x)+(1-\lambda)f(x_{0})=f(x_{0})+\lambda(f(x_{0}+\mu x)-f(x_{0}))$. 

Pick $\lambda>0$ sufficiently small to have $f(x_{0})+\lambda(f(x_{0}+\mu x)-f(x_{0}))\le a$.
Hence $x_{0}+\lambda\mu x\in[f\le a]$, that is, $[f\le a]-x_{0}$
is absorbing. This implies $[f\le a]\in\mathscr{V}_{\tau}(x_{0})$
and so $x_{0}\in\operatorname*{int}_{\tau}(\operatorname*{dom}f)$.
It is clear that $f$ is bounded above on $[f\le a]$. 

(iii) $\Rightarrow$ (i) If $B\subset X^{*}$ is pointwise-bounded
then $\operatorname*{dom}\sigma_{B}=X$ and $\sigma_{B}\in\Gamma_{w}(X)$.
Thus $\sigma_{B}$ is $\tau-$continuous at $0$, i.e., there exist
a symmetric $V\in\mathscr{V}_{\tau}(0)$ and $M<\infty$ such that
$\sigma_{B}(x)\le M$, for every $x\in V$. This comes to $B\subset(\tfrac{1}{M}V)^{\circ}$,
that is, $B$ is $\tau-$equicontinuous. \end{proof} 

\begin{remark} \emph{\label{BS-Bar} When $(X,\tau)$ is a locally
convex space (LCS for short), the closed convex sets in the $\tau$
and weak topologies on $X$ coincide. In this case the Banach-Steinhaus
property comes to the fact that $(X,\tau)$ is }barreled\emph{, that
is, every absorbing, convex, and $\tau-$closed subset of $X$ is
a $\tau-$neighborhood of $0\in X$. Equivalently, every $f\in\Gamma_{\tau}(X)$
is $\tau-$continuous on $\operatorname*{int}_{\tau}(\operatorname*{dom}f)$,
in which case, $\operatorname*{int}_{\tau}(\operatorname*{dom}f)=\operatorname*{core}(\operatorname*{dom}f)$.}
\end{remark} 

For every TVS $(X,\tau)$, let us denote by $\tau^{\circ}$ the weakest
local convex topology on $X$ which is compatible with the duality
$(X,X^{*})$ and finer than $\tau$. In case $\tau^{\circ}$ exists,
$\tau$ and $\tau^{\circ}$ share the equicontinuous sets of $X^{*}$
and, in general, all the properties relying on equicontinuity or duality
do not distinguish themselves between $(X,\tau)$ and $(X,\tau^{\circ})$.
From this point of view it is the same if we consider the TVS $(X,\tau)$
or its associated LCS $(X,\tau^{\circ})$. 

In the pathological cases when $\tau^{\circ}$ does not exist, e.g.
when $X^{*}=\{0\}$ and $(X,\tau)$ is (Hausdorff) separated, the
operator local boundedness is trivially verified.

\section{The local boundedness theorem}

Let us note that in every TVS $X$ there exist maximal monotone operators
$T:X\rightrightarrows X^{*}$ such that $\operatorname*{int}\operatorname*{Pr}_{X}(\operatorname*{dom}\varphi_{T})$
is non-empty. For example $T=X\times\{0\}$ has $\varphi_{T}(x,x^{*})=\iota_{\{0\}}(x^{*})$,
$(x,x^{*})\in X\times X^{*}$ and $\operatorname*{dom}\varphi_{T}=X\times\{0\}$.
This example is the most general possible since there exist non-trivial
TVS's $X$ for which $X^{*}=\{0\}$.

Our next result proves that the Banach-Steinhaus property is the minimal
context condition under which the local boundedness of an operator
with proper Fitzpatrick function holds. 

\begin{theorem} \label{LBG} Let $X$ be a TVS. The following are
equivalent:

\emph{(i)} $X$ has the Banach-Steinhaus property;

\emph{(ii)} Every $T:X\rightrightarrows X^{*}$ is locally bounded
on $\operatorname*{core}\operatorname*{Pr}_{X}(\operatorname*{dom}\varphi_{T})$;

\emph{(iii)} Every $T:X\rightrightarrows X^{*}$ is locally bounded
on $\operatorname*{int}\operatorname*{Pr}_{X}(\operatorname*{dom}\varphi_{T})$.

\end{theorem}

\begin{proof} (i) $\Rightarrow$ (ii) see the published version.

(ii) $\Rightarrow$ (iii) is straightforward.

(iii) $\Rightarrow$ (i) Let $B\subset X^{*}$ be pointwise-bounded
and let $T:X\rightrightarrows X^{*}$ be such that $\operatorname*{Graph}T=\{0\}\times B$.
Then $\operatorname*{dom}\varphi_{T}=X\times X^{*}$ since $\varphi_{T}(x,x^{*})=\sigma_{B}(x)<\infty$,
for every $x\in X$, $x^{*}\in X^{*}$. The local boundedness of $T$
at $0$ shows that $T0=B$ is equicontinuous, i.e., $X$ has the Banach-Steinhaus
property. \end{proof}

\strut

As previously seen in Remark \ref{BS-Bar}, when  $(X,\tau)$ is a
LCS, condition (i) in Theorem \ref{LBG} can be equivalently rephrased
as $(X,\tau)$ is barreled. 

\medskip

In the previous result, only the operators which have a proper Fitzpatrick
function are interesting. We denote this class by

\begin{center}
$\mathscr{P}(X):=\{T:X\rightrightarrows X^{*}\mid\operatorname*{Graph}T\neq\emptyset,\ \operatorname*{dom}\varphi_{T}\neq\emptyset\}$. 
\par\end{center}

In the literature, the most used class of operators that have a proper
Fitzpatrick function is the class of non-empty monotone operators
(see e.g. \cite{MR2207807,MR2389004,MR2453098,astfnc,MR2583911,MR2594359,MR2577332}),
but, more generally, it is easily checked that $\mathscr{M}_{\infty}(X)\subset\mathscr{P}(X)$
since $\operatorname*{Graph}T\subset\operatorname*{dom}\varphi_{T}$,
for every $T\in\mathscr{M}_{\infty}(X)$. 

\begin{definition} \emph{Given $(X,\tau)$ a TVS, for every $T:X\rightrightarrows X^{*}$
we denote by\[
\Omega_{T}:=\{x\in\overline{D(T)}\mid T\ {\rm is\ locally\ bounded\ at}\ x\}\]
the }(meaningful) local boundedness set of\emph{ $T$.} \end{definition}

In this notation, Theorem \ref{LBG} states that the TVS $(X,\tau)$
has the Banach-Steinhaus property iff\begin{equation}
\operatorname*{core}\operatorname*{Pr}\,\!\!_{X}(\operatorname*{dom}\varphi_{T})\cap\overline{D(T)}\subset\Omega_{T},\ \forall T\in\mathscr{P}(X)\label{eq:lbg}\end{equation}
iff\begin{equation}
\operatorname*{int}\operatorname*{Pr}\,\!\!_{X}(\operatorname*{dom}\varphi_{T})\cap\overline{D(T)}\subset\Omega_{T},\ \forall T\in\mathscr{P}(X).\label{eq:lbg-int}\end{equation}

On one hand, one cannot expect that, for every $T\in\mathscr{P}(X)$,
the inclusions in (\ref{eq:lbg}) or (\ref{eq:lbg-int}) to be equalities,
since there exist operators $T$ which are locally bounded at some
$x\in D(T)$ but $x\not\in\operatorname*{core}\operatorname*{Pr}_{X}(\operatorname*{dom}\varphi_{T})$
simply because $\operatorname*{core}\operatorname*{Pr}_{X}(\operatorname*{dom}\varphi_{T})$
is empty. Indeed, take $X$ a Hilbert space, $V\in\mathscr{V}(0)$,
$M\subset X$ a proper closed subspace, and $T:D(T)=M\subset X\rightrightarrows X^{*}$,
$Tx=\{0\}$, if $x\in M\cap V$; $Tx=M^{\perp}$, $x\in M\setminus V$.
Then $\varphi_{T}=\iota_{M\times M^{\perp}}$ and so $\operatorname*{Pr}_{X}(\operatorname*{dom}\varphi_{T})=M$
and $\operatorname*{core}\operatorname*{Pr}_{X}(\operatorname*{dom}\varphi_{T})=\emptyset$. 

Therefore, for some $T\in\mathscr{P}(X)$, the algebraic (or topological)
interior of $\operatorname*{Pr}\,\!\!_{X}(\operatorname*{dom}\varphi_{T})$
is not the perfect description for $\Omega_{T}$. 

On the other hand, for $X$ a Banach space and $T\in\mathfrak{M}(X)$
with $\overline{D(T)}$ convex \begin{equation}
\operatorname*{int}\operatorname*{Pr}\,\!\!_{X}(\operatorname*{dom}\varphi_{T})=\operatorname*{core}\operatorname*{Pr}\,\!\!_{X}(\operatorname*{dom}\varphi_{T})=\operatorname*{int}D(T)=\Omega_{T},\label{ot-mm-int}\end{equation}
(see \cite[Theorem\ 3.11.15,\ p.\ 286]{MR1921556}, \cite{MR0253014},
and \cite[Lemma 41]{mmicq}), that is, for this particular class of
operators and type of space the problem of perfectly describing the
local boundedness set is solved. 

These two points of view prove that the general description of $\Omega_{T}$,
given in Theorem \ref{LBG} and provided for all operators $T\in\mathscr{P}(X)$,
cannot be further improved.

We conjecture, that, under the assumption that $X$ has the Banach-Steinhaus
property, \begin{equation}
\operatorname*{int}\operatorname*{Pr}\,\!\!_{X}(\operatorname*{dom}\varphi_{T})\cap\overline{D(T)}=\operatorname*{core}\operatorname*{Pr}\,\!\!_{X}(\operatorname*{dom}\varphi_{T})\cap\overline{D(T)},\label{conj}\end{equation}
for every $T\in\mathscr{P}(X)$ (or a suitable subclass such as $\mathfrak{M}(X)$).
A partial answer to this conjecture follows next. 

\begin{proposition}\label{int-nb} Let $X$ be a barrelled normed
space and let $T:X\rightrightarrows X^{*}$. Then $\operatorname*{int}\operatorname*{Pr}_{X}(\operatorname*{dom}\varphi_{T})=\operatorname*{core}\operatorname*{Pr}_{X}(\operatorname*{dom}\varphi_{T})$.
\end{proposition}

\begin{proof} Assuming that $\operatorname*{core}\operatorname*{Pr}_{X}(\operatorname*{dom}\varphi_{T})\neq\emptyset$
and because $X$ is a barrelled normed space, we get the conclusion
(see e.g. \cite[Proposition\ 2.7.2\ (vi),\ p. 116]{MR1921556}). \end{proof}

\section{Local boundedness on subclasses }

This section deals with the validity of the implications (ii) $\Rightarrow$
(i) or (iii) $\Rightarrow$ (i) in Theorem \ref{LBG} on subclasses
of operators and subsets of local boundedness. 

\begin{proposition} \label{ptw}Let $(X,\tau)$ be a LCS and let
$B\subset X^{*}$ be pointwise-bounded. Then $\{0\}\times B$ admits
a (maximal) monotone extension $T:X\rightrightarrows X^{*}$ with
$0\in\operatorname*{core}(\operatorname*{conv}D(T))$. \end{proposition} 

\begin{proof} see the published version. \end{proof}

\begin{theorem} \label{LBM} Let $(X,\tau)$ be a TVS. For every
$\mathscr{C}\in\{\mathfrak{B}(X),\mathscr{G}(X),\mathcal{M}(X)\}\cup\{\mathscr{G}_{\epsilon}(X),\mathcal{M}_{\epsilon}(X),\mathfrak{M}_{\epsilon}(X),\mathcal{M}_{\epsilon}^{+}(X)\mid\epsilon>0\}$
the following are equivalent

\medskip

\emph{(i)} $(X,\tau)$ has the Banach-Steinhaus property,

\medskip

\emph{(ii)} $\operatorname*{core}\operatorname*{Pr}_{X}(\operatorname*{dom}\varphi_{T})\cap\overline{D(T)}\subset\Omega_{T}$,
$\forall T\in\mathscr{C}$,

\medskip

\emph{(iii)} $\operatorname*{int}\operatorname*{Pr}_{X}(\operatorname*{dom}\varphi_{T})\cap\overline{D(T)}\subset\Omega_{T}$,
$\forall T\in\mathscr{C}$. \end{theorem}

\begin{proof} (ii) $\Rightarrow$ (iii) is plain. (i) $\Rightarrow$
(ii) holds due to Theorem \ref{LBG} because every considered class
of operators $\mathscr{C}$ is a subclass of $\mathscr{P}(X)$. More
precisely, if $M\in\mathcal{M}_{\epsilon}^{+}(X)$, for some $\epsilon\ge0$,
then, $M=T_{\epsilon}^{+}$, for some $T\in\mathcal{M}_{\epsilon}(X)$.
We have $T\subset[\varphi_{M}\le c+\epsilon]\subset\operatorname*{dom}\varphi_{M}$;
whence $M\in\mathscr{P}(X)$ and so $\bigcup_{\epsilon\ge0}\mathcal{M}_{\epsilon}^{+}(X)\subset\mathscr{P}(X)$. 

(iii) $\Rightarrow$ (i) For every pointwise bounded $B\subset X^{*}$
we show that $\{0\}\times B$ admits an extension $T$ which belongs
to every class considered and $\operatorname*{Pr}_{X}(\operatorname*{dom}\varphi_{T})$
$=X$. In this case the local boundedness of $T$ at $0$ proves that
$B$ is equicontinuous.

First, note that, whenever $B\subset X^{*}$ is pointwise bounded,
$\operatorname*{dom}\sigma_{B}=X$ and from $\varphi_{\partial\sigma_{B}}(x,x^{*})\le\sigma_{B}(x)+\iota_{B}(x^{*})$,
for every $(x,x^{*})\in X\times X^{*}$, one gets $X\times B\subset\operatorname*{dom}\varphi_{\partial\sigma_{B}}$
and $\operatorname*{Pr}_{X}(\operatorname*{dom}\varphi_{\partial\sigma_{B}})=X$.
Therefore $\partial\sigma_{B}\in\mathfrak{B}(X)$ fulfills all the
required conditions. Since $\mathfrak{B}(X)\subset\mathscr{G}(X)\subset\mathcal{M}(X)$,
this example completes the argument for the classes $\mathfrak{B}(X),\mathscr{G}(X),\mathcal{M}(X)$. 

Also, since $\sigma_{B}\in\Gamma_{\tau}(X)$ and $\operatorname*{dom}\sigma_{B}=X$,
$D(\partial_{\epsilon}\sigma_{B})=X$, for every $\epsilon>0$; whence
$\partial_{\epsilon}\sigma_{B}\in\mathscr{G}_{\epsilon}(X)$ has the
required properties. Since $\mathscr{G}_{\epsilon}(X)\subset\mathcal{M}_{\epsilon}(X)$,
$D(T)=X$, for every extension $T$ of $\partial_{\epsilon}\sigma_{B}$,
and $\mathfrak{M}_{\epsilon}(X)\subset\mathcal{M}_{\epsilon}^{+}(X)$
this example proves the implication for the classes $\{\mathscr{G}_{\epsilon}(X),\mathcal{M}_{\epsilon}(X),\mathfrak{M}_{\epsilon}(X),\mathcal{M}_{\epsilon}^{+}(X)\mid\epsilon>0\}$.
\end{proof}

\begin{theorem} \label{LBMax} Let $(X,\tau)$ be a LCS. For every
$\mathscr{C}\in\{\mathfrak{M}(X),\mathcal{M}^{+}(X)\}$ the following
are equivalent

\medskip

\emph{(i)} $(X,\tau)$ is barreled,

\medskip

\emph{(ii)} $\operatorname*{core}\operatorname*{Pr}_{X}(\operatorname*{dom}\varphi_{T})\cap\overline{D(T)}\subset\Omega_{T}$,
$\forall T\in\mathscr{C}$. \end{theorem}

\begin{proof} Since $\mathfrak{M}(X)\cup\mathcal{M}^{+}(X)\subset\mathscr{P}(X)$,
(i) $\Rightarrow$ (ii) holds due to Theorem \ref{LBG}. 

We have seen in Proposition \ref{ptw}, that $\{0\}\times B$ admits
a maximal monotone extension $T:X\rightrightarrows X^{*}$ with $0\in\operatorname*{core}(\operatorname*{conv}D(T))\subset\operatorname*{core}\operatorname*{Pr}_{X}(\operatorname*{dom}\varphi_{T})$.
Together with $\mathfrak{M}(X)\subset\mathcal{M}^{+}(X)$ we infer
that (ii) $\Rightarrow$ (i) is true for $\mathscr{C}\in\{\mathfrak{M}(X),\mathcal{M}^{+}(X)\}$.
\end{proof}

\begin{theorem} \label{LBM-conv} Let $(X,\tau)$ be a TVS. For every
$\mathscr{C}\in\{\mathscr{G}_{\epsilon}(X),\mathcal{M}_{\epsilon}(X),\mathfrak{M}_{\epsilon}(X),$
$\mathcal{M}_{\epsilon}^{+}(X)\mid\epsilon>0\}$ the following are
equivalent

\medskip

\emph{(i)} $(X,\tau)$ has the Banach-Steinhaus property,

\medskip

\emph{(ii)} $\operatorname*{core}(\operatorname*{conv}D(T))\cap\overline{D(T)}\subset\Omega_{T}$,
$\forall T\in\mathscr{C}$. 

\medskip

\emph{(iii)} $\operatorname*{int}(\operatorname*{conv}D(T))\cap\overline{D(T)}\subset\Omega_{T}$,
$\forall T\in\mathscr{C}$. \end{theorem}

\begin{proof} (ii) $\Rightarrow$ (iii) is plain. (i) $\Rightarrow$
(ii) For $\mathscr{C}\in\{\mathscr{G}_{\epsilon}(X),\mathcal{M}_{\epsilon}(X),\mathfrak{M}_{\epsilon}(X)\mid\epsilon>0\}$
we use Theorem \ref{LBG}, because all these classes are subclasses
of $\mathscr{M}_{\infty}(X)$ and $D(T)\subset\operatorname*{Pr}_{X}(\operatorname*{dom}\varphi_{T})$,
for every $T\in\mathscr{M}_{\infty}(X)$. 

For $\mathscr{C}=\mathcal{M}_{\epsilon}^{+}(X)$, let $M\in\mathcal{M}_{\epsilon}^{+}(X)$,
i.e., $M=T_{\epsilon}^{+}$, for some $T\in\mathcal{M}_{\epsilon}(X)$.
Since $T\in\mathcal{M}_{\epsilon}(X)$, $T\subset M$ and $T\subset\operatorname*{Pr}_{X}(\operatorname*{dom}\varphi_{M})$.
We apply Theorem \ref{LBG} for $M$ to get that $M$ and implicitly
$T$ are local bounded on $\operatorname*{core}\operatorname*{Pr}_{X}(\operatorname*{dom}\varphi_{M})\supset\operatorname*{core}(\operatorname*{conv}D(T))$. 

(iii) $\Rightarrow$ (i) For every pointwise bounded $B\subset X^{*}$
we claim that $\{0\}\times B$ admits an extension $T$ which belongs
to every class considered and $D(T)=X$. In this case the local boundedness
of $T$ at $0$ proves that $B$ is equicontinuous.

Indeed, for $\epsilon>0$, $\partial_{\epsilon}\sigma_{B}$ is an
extension of $\{0\}\times B$ which belongs to $\mathscr{G}_{\epsilon}(X)$
and $D(\partial_{\epsilon}\sigma_{B})=X$. Since $\mathscr{G}_{\epsilon}(X)\subset\mathcal{M}_{\epsilon}(X)$
, $\mathfrak{M}_{\epsilon}(X)\subset\mathcal{M}_{\epsilon}^{+}(X)$,
and any extension $T$ of $\partial_{\epsilon}\sigma_{B}$ has $D(T)=X$,
this example completes the argument for $\{\mathscr{G}_{\epsilon}(X),\mathcal{M}_{\epsilon}(X),\mathfrak{M}_{\epsilon}(X),\mathcal{M}_{\epsilon}^{+}(X)\mid\epsilon>0\}$.
\end{proof} 

\begin{theorem} \label{LBMax-conv} Let $(X,\tau)$ be a LCS. For
every $\mathscr{C}\in\{\mathfrak{M}(X)$,$\mathcal{M}(X)$,$\mathcal{M}^{+}(X)\}$
the following are equivalent

\medskip

\emph{(i)} $(X,\tau)$ is barreled,

\medskip

\emph{(ii)} $\operatorname*{core}(\operatorname*{conv}D(T))\cap\overline{D(T)}\subset\Omega_{T}$,
$\forall T\in\mathscr{C}$. 

\end{theorem}

\begin{proof} From Theorem \ref{LBG}, (i) $\Rightarrow$ (ii) holds
since $\mathcal{M}(X)\cup\mathcal{M}^{+}(X)\subset\mathscr{P}(X)$. 

Because $\{0\}\times B$ admits a maximal monotone extension $T:X\rightrightarrows X^{*}$
with $0\in\operatorname*{core}(\operatorname*{conv}D(T))$ (see Proposition
\ref{ptw}), (ii) $\Rightarrow$ (i) is true for $\mathfrak{M}(X)$,
followed by its super-classes $\mathcal{M}(X)$ and $\mathcal{M}^{+}(X)$.
\end{proof} 

\strut

The following theorem is a broad generalization of the main result
in \cite{MR995141} and of  \cite[Theorem 4.2]{MR2675660}. 

\begin{theorem} \label{LBM-conv-plus} Let $(X,\tau)$ be a LCS and
let $\epsilon\ge0$. The following are equivalent:

\medskip

\emph{(i)} $X$ is barreled,

\medskip

\emph{(ii)} For every $T\in\mathcal{M}_{\epsilon}(X)$, $T_{\epsilon}^{+}$
is locally bounded on $\operatorname*{core}(\operatorname*{conv}D(T))$;

\medskip

\emph{(iii)} For every $T\in\mathfrak{M}_{\epsilon}(X)$, $T$ is
locally bounded on $\operatorname*{core}(\operatorname*{conv}D(T))$.
\end{theorem} 

\begin{proof} (i) $\Rightarrow$ (ii) follows from Theorem \ref{LBG}
because $D(T)\subset\operatorname*{Pr}_{X}(\operatorname*{dom}\varphi_{T_{\epsilon}^{+}})$,
for every $\epsilon\ge0$, $T\in\mathcal{M}_{\epsilon}(X)$. 

(ii) $\Rightarrow$ (iii) Whenever $T\in\mathfrak{M}_{\epsilon}(X)$,
$T=T_{\epsilon}^{+}$. 

(iii) $\Rightarrow$ (i) We use Theorems \ref{LBM-conv}, \ref{LBMax-conv}
to conclude. \end{proof}

\bibliographystyle{plain}

\end{document}